\documentclass[journal,twoside,web]{ieeecolor_no_firstpage_logo}
\usepackage{generic}
\usepackage{cite}
\usepackage{amsmath,amssymb,amsfonts}
\usepackage{mathtools}
\usepackage{algorithmic}
\usepackage{graphicx}
\usepackage{algorithm,algorithmic}
\usepackage{hyperref}
\hypersetup{hidelinks=true}
\usepackage{textcomp}
\usepackage{tabularx}
\usepackage{physics}
\usepackage{booktabs}
\usepackage{xcolor}
\usepackage{listings}
\usepackage{subcaption}

\DeclareCaptionFont{eightpt}{\fontsize{8}{9.6}\selectfont}

\usepackage{makecell}
\usepackage{ragged2e}

\newtheorem{definition}{Definition}
\newtheorem{theorem}{Theorem}
\newtheorem{corollary}{Corollary}
\newtheorem{remark}{Remark}

\def\BibTeX{{\rm B\kern-.05em{\sc i\kern-.025em b}\kern-.08em
    T\kern-.1667em\lower.7ex\hbox{E}\kern-.125emX}}
\begin{document}
\title{A Control-Theoretic Approach for Resource-Aware Consensus in Multi-Agent AI}
\author{James Flagg and Esteban A. Hernandez-Vargas 
\thanks{Submitted on Aug. 17, 2026. Supported by NSF Grant No.\ DMS-2315862.}
\thanks{James Flagg is with the Dept.\ of Mathematics and Statistical Science, University of Idaho, Moscow, ID 83843 USA. (e-mail: flag1779@vandals.uidaho.edu). }
\thanks{Esteban A. Hernandez-Vargas is with the Dept.\ of Mathematics and Statistical Science, University of Idaho, Moscow, ID 83843 USA. (e-mail: esteban@uidaho.edu).}
}

\maketitle

\begin{abstract}
Large language model multi-agent systems (LLM-MAS) rely on inter-agent communication to solve complex reasoning tasks, yet rigorous guarantees relating consensus performance to computational resources remain limited. 
Here, we present a novel way to characterize collective belief dynamics as a discrete-time switched system in which communication topologies have distinct consensus-contraction rates and token costs. By augmenting the belief dynamics with the remaining computational budget, we define a consensus safe set that jointly captures agreement and resource feasibility. We derive explicit bounds on consensus time and token expenditure and construct a consensus-budget certificate region guaranteeing finite-time convergence without resource exhaustion. 
We further establish conditions under which adaptive topology switching achieves a trade-off between convergence speed and communication cost relative to fixed-topology strategies. Numerical experiments and live LLM-MAS deployments show the predicted consensus-cost trade-offs, demonstrating how control-theoretic certificates can enable resource-aware coordination in AI systems.
\end{abstract}


\section{Introduction}
\label{sec:introduction}
\IEEEPARstart{L}{arge} Language Models (LLMs) are transforming society as rapidly as they are evolving technically. 
The base unit of measurement for LLM computational cost is the token \cite{bib4}. Tokens are vectors which encode inputs, outputs, and thought processes \cite{bib5}. 
Token expenditures for Chat-GPT alone, though not publicly disclosed, could reach into the quadrillions as Chat-GPT processes 2.5 billion prompts every day \cite{bib6}. Experts document repercussions from these high computational costs. Training GPT-5 required 1 billion liters of water \cite{bib6}. AI electricity consumption in 2025 reached 93TWh, enough to supply Nigeria for two years \cite{bib6}. The data centers which power LLMs are taking their position among nations in the global economy. By 2030, data centers are projected to rank 6th globally in electricity consumption, just behind Japan \cite{bib6}. The projected water use of data centers, amounting to 9.3 trillion liters, could meet the domestic needs of Sub-Saharan Africa's 1.3 billion population. 

\begin{table*}[ht!]
\centering
\caption{Comparative Summary of Frontier Large Language Models (2026)}
\label{tab:model_comparison}
\small
\begin{tabularx}{\textwidth}{l X X c}
\toprule
\textbf{Model} & \textbf{Technical Focus \& Capabilities} & \textbf{Major Applications} & \textbf{Benchmarks} \\
& & & {\scriptsize (GPQA / SWE)} \\
\midrule

\makecell[l]{\textbf{GPT-5.5} \\ {\scriptsize (OpenAI)}} & 
The GPT-5 generation performs 3x better than previous models in the area of sycophancy, enabling more effective collaboration among agents \cite{bib7}. & 
Wide range of personal and corporate tasks, including coding and writing & 
93.18\% \newline { / 82.6\%} \\
\addlinespace

\makecell[l]{\textbf{Gemini 3.1 Pro} \\ {\scriptsize (Google)}} & 
With a 1M token context-window, Gemini 3.1 specializes in solving large problems in a cross-modal context \cite{bib8}. & 
Tasks involving video and audio & 
95.45\% \newline { / 78.8\%} \\
\addlinespace

\makecell[l]{\textbf{Claude Opus 4.8} \\ {\scriptsize (Anthropic)}} & 
Dynamic workflows deploy multiple agents in parallel to complete large tasks from planning to execution \cite{bib9}. & 
Agentic coding & 
92.42\% \newline { / 88.6\%} \\
\addlinespace

\makecell[l]{\textbf{DeepSeek R1} \\ {\scriptsize (DeepSeek)}} & 
Advanced problem-solving capabilities stem from an RL (reinforcement learning)  framework which rewards correct outcomes achieved via explorative reasoning processes \cite{bib10}. & 
Cost-effective math and logic & 
71.5\% \newline { / 49.2\%} \\
\addlinespace

\makecell[l]{\textbf{LLaMa 4 Mav.} \\ {\scriptsize (Meta)}} & 
Mixture-of-Experts (MoE) framework delegates sub-layers of agents to analyze up to 10M tokens \cite{bib11}. & 
LLaMa’s weights are publicly accessible, enabling reproducibility in simulations. & 
67.1\% \newline { / 21.04\%} \\

\bottomrule
\end{tabularx}
\end{table*}



These resource demands become particularly consequential in agentic AI, where solving a single task may require repeated LLM inference, tool calls, and communication across multiple reasoning cycles.
This shift from one-shot inference toward iterative autonomous reasoning is exemplified by the \emph{ReAct} framework, which advanced LLMs to autonomous agents through a closed-loop reasoning-acting cycle \cite{bib16}. First, the model reasons about the current state. Based on its assessment, the model issues a tool call. Finally, the model observes the results before restarting the loop.
Multi-Agent Systems (MAS) extend this paradigm by coordinating multiple LLM agents to complete long and complex tasks. These systems build upon and scale the core ideas of ReAct, incorporating new architectures for inter-agent relationships and maintaining persistent memory across the team \cite{bib17}.



Reaching \textit{consensus} is the goal of coordination for MAS. When agents in the system achieve consensus, their individual belief states align. In practice, consensus is achieved when the difference between individual belief and the collective average is less than a given tolerance $\varepsilon >0.$ While we use the term MAS to refer specifically to teams of LLMs, it should be noted that there are many types of Multi-Agent Systems, spanning a variety of contexts. MAS consensus is a wide field of research with deep classical roots 
\cite{bib19,bib20,bib21}. 

Despite rapid empirical progress, the theoretical foundations of agentic AI remain underdeveloped \cite{bib22,bib23,bib24}. As agentic systems form the basis for an increasing range of applications, from software development to policy simulation \cite{bib28}, the development of a theoretical framework for robust safety constraints is correspondingly necessary. Unsafe behaviors in MAS include hallucinations, infinite reasoning regressions, and high computational costs due to repeated tool calls and long generations. 
Amodei et al.~\cite{bib29} enumerate concrete failure
modes---reward misspecification, distributional shift, safe
exploration, and robustness to distributional shift---and call for
formal methods to address them.
Dalrymple et al.~\cite{bib30} articulate a framework
for ``guaranteed safe AI'' requiring machine-checkable safety proofs,
but note that no general formal model for LLM-based agents yet exists
at the required level of rigor.

In Table \ref{tab:model_comparison}, we summarize the distinguishing features of the frontier LLM platforms, with a special interest in their aptitude for consensus-driven MAS. Adherence to academically standard metrics for LLM performance is essential when evaluating a model. The GPQA Diamond database contains 448 multiple-choice questions in biology, physics, and chemistry; PhD experts average 65\% accuracy, while non-experts with internet access achieve 34\% \cite{bib25}. The GPQA Diamond score evaluates a model's reasoning and logic skills. To measure a model's ability to carry out a task, an essential skill for agentic coding, we use the SWE-bench score. Jimenez et al. introduced their testing framework of 2,294 software engineering problems in 2024 \cite{bib26}. Because resolving bugs requires complex interaction with an agent's environment, nuanced analysis of large repositories, and involved planning, the SWE-bench is an accurate marker of an agent's viability. The benchmark scores are sourced from the frontier rankings on the Vals AI Leaderboard \cite{bib27}. 

Individual LLMs operate as open-loop systems, mapping an input prompt to an output token sequence without internal monitoring or feedback. An LLM-MAS, on the other hand, self-evaluates at every iteration of a closed-loop workflow. The harness, which is the architecture that unifies the agents, acts as a \textit{controller} that receives internal assessments and adjusts the system trajectory accordingly. These internal assessments could be the number of failed unit tests (in the context of agentic coding), a transcript which an ``evaluator" LLM generates, or simply the internal belief states of the agents in the system. The MAS evolves in discrete time, where each cycle of the workflow increments a discrete-time index which chronologizes policy changes, completion of communication rounds, and input from the harness.

Zhang et al. establish how the harness architecture unites multiple agents into a dynamical system, demonstrating how a change in the harness can greatly effect overall performance \cite{bib31}. Recent work has led to the development of many harness architectures with specific optimization goals. The VeRO harness enables agents to improve other agents by modifying their source code, working within a computational budget \cite{bib32}.
Due to the complexity of MAS, current research focuses on a subset of the system state to model as a dynamical system. For instance, Eslami et al. focus on runtime agency and interactions between system modifications, characterizing stability as delay times beneath a certain threshold \cite{bib33}. 

\textit{Contributions:} This paper develops a control-theoretic framework for resource-aware consensus in LLM-based multi-agent systems. 
We incorporate token budget directly into the system state and study whether a multi-agent reasoning process can be certified to reach practical consensus before exhausting its available resources. The main contributions are as follows.

\textit{i)} LLM-MAS coordination is formulated as a discrete-time switched system in which each mode represents a communication topology with an associated consensus-contraction rate and token cost. The harness acts as a supervisory controller that selects the active communication topology according to the collective disagreement among agents.

\textit{ii)} A consensus-budget certificate region is constructed to provide a sufficient condition, in terms of initial disagreement and available budget, for finite-time convergence to $\varepsilon$-consensus without resource exhaustion. We further establish conditions under which adaptive topology switching achieves a trade-off between the convergence speed of dense communication and the lower cost of sparse communication.

\textit{iii)} The theoretical predictions are evaluated through numerical simulations and live LLM-MAS deployments in a software-architecture planning task. Adaptive, dense, and sparse communication strategies are compared to characterize the interplay among disagreement contraction, communication topology, and computational cost.

To the best of our knowledge, this is the first control-theoretic formulation that jointly models consensus and computational-budget dynamics in LLM-MAS and provides explicit certificates linking communication topology, convergence, and token expenditure.




\section{Preliminaries and Problem Formulation}
Throughout, $\mathbb{R}$ denotes the field of real numbers. $\mathbb{R}^d$ denotes the set of vectors of dimension $d$ whose entries are real numbers. Similarly, $\mathbb{R}^{N\times N}$ denotes the set of $N \times N$ matrices whose entries are real numbers. $\mathbb{Z}$ denotes the ring of integers, while $\mathbb{Z}_{\geq 0} = \{ x \in \mathbb{Z} : x \geq 0 \}$ denotes the subset of nonnegative integers. $\mathbb{N}$ is the set of natural numbers $\{1,2,3,...\}$. For a matrix $M$, $M^\top$ represents the transpose of $M.$ We denote the $N$-dimensional column vector of ones in bold: $\textbf{1} \in \mathbb{R}^N$. In a directed graph $\mathcal{G}$ with edge-set $E$ and vertices-set $V$, we use $V \times V = \{ (v,u) | u \in V \text{ and } v \in V\} $ to denote the set of ordered pairs where each pair corresponds to an edge from vertex $v$ to vertex $u.$

We use three norms in the framework that follows. For a vector $v \in \mathbb{R}^d,$ we write $\left\| v\right\|_2$ to denote the standard Euclidean distance norm, defined as:
$$
\left\| v\right\|_2 = \sqrt{\sum_{i = 1}^dv_i^2}.
$$
For a matrix $M \in \mathbb{R}^{N \times N}$, the Frobenius norm $\left\| M \right\|_F$ measures the overall size of the  matrix. It is an extension to matrices of the Euclidean norm, defined as:
$$
\left\| M \right\|_F = \sqrt{\sum_{i=1}^N\sum_{j=1}^Na_{ij}^2}.
$$
The spectral norm of a matrix $M \in \mathbb{R}^{N \times N}$ quantifies how much $M$ can stretch or amplify a vector $x \in \mathbb{R}^N:$ 
$$
\left\| M \right\|_2 = \max_{x \neq 0} \frac{\left\| M x\right\|_2}{\left\| x \right\|_2}.
$$
Notice that the norms on the right hand side of this equivalence are both Euclidean, since $Mx$ is a column vector of dimension $N.$
For an action matrix $W$ and subspace $\textbf{s}$, we write $W\big|_{\textbf{s}}$ to show that we are restricting $W$ to $\textbf{s}$. For $a \in \mathbb{R},$ we write $\lceil a \rceil$ to apply the ceiling function to $a.$

\begin{figure}
    \centering
    \includegraphics[width=1\linewidth]{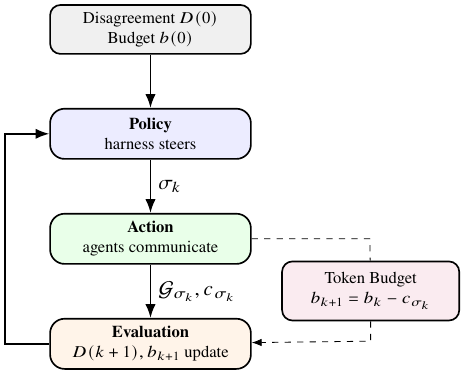}
    \caption{MAS communication loop as a hybrid switched feedback control system. The harness selects a communication topology (switching signal $\sigma_k$) based on disagreement level $D(k)$ and remaining token-budget $b_k$; the agents debate according to graph $\mathcal{G}_{\sigma_k}$, accruing computational cost $c_{\sigma_k}$; evaluation of the updated disagreement $D(k+1)$ and budget $b_{k+1}$ drives the next policy. Token consumption is dictated by the communication topology.}
    \label{fig:fig1}
\end{figure}

\subsection{Problem Statement}

To model consensus-budget dynamics, we exploit a key architectural feature of LLM-MAS: the harness can regulate inter-agent communication by activating different communication topologies during the reasoning process. We therefore model the MAS as a discrete-time switched system, where the discrete-time index $k$ represents communication rounds and the switching signal $\sigma_k$ selects the active communication topology.

Under mode $\sigma_k$, the communication topology is represented by the directed graph $\mathcal{G}_{\sigma_k}=(V,E_{\sigma_k})$ and its associated weight matrix $W_{\sigma_k}$, which determines how agents incorporate the belief states of their neighbors. Different communication modes induce different rates of disagreement contraction and incur different computational costs. The harness acts as a supervisory controller that selects $\sigma_k$ according to the aggregate disagreement $D(k)$, thereby steering the evolution of the collective belief state. The resulting closed-loop architecture is illustrated in Fig.~\ref{fig:fig1}.

\begin{figure}
    \centering
    \includegraphics[width=1\linewidth]{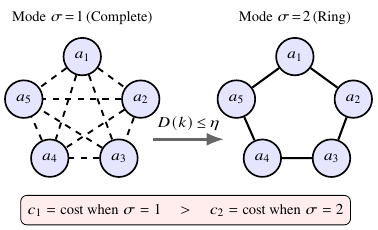}
    \caption{Five LLM agents under norm-threshold switching. Mode~1 (complete $K_5$, costlier) is activated when the aggregate disagreement satisfies $D(k)>\eta$; Mode~2 (ring $C_5$, cheaper) is activated when $D(k)\leq\eta$.}
    \label{fig:fig2}
\end{figure}

The control objective is to drive the MAS to practical consensus, $D(k)\leq\varepsilon$, without exhausting the initial token budget $b_0$. This introduces a fundamental trade-off: dense communication topologies may contract disagreement rapidly but incur higher token costs, whereas sparse topologies reduce communication cost at the expense of slower convergence. We address this trade-off through norm-threshold switching: a dense communication mode is activated when $D(k)>\eta$, while a lower-cost sparse mode is used once $D(k)\leq\eta$, as illustrated in Fig.~\ref{fig:fig2}.

Accordingly, the problem addressed in this paper is to determine conditions under which the switching policy guarantees that the MAS reaches an $\varepsilon$-consensus state before its computational budget is exhausted. To this end, we augment the collective belief state with the remaining budget and define the $\varepsilon$-consensus safe set $\Omega_\varepsilon$. We then characterize a consensus-budget certificate region $\Omega_{\mathrm{cert}}$ whose states possess sufficient computational resources to guarantee finite-time reachability of $\Omega_\varepsilon$.

\section{Multi-agent Systems Consensus Framework}\label{sec:model}
We work under the following assumptions: (i) \textit{synchronous activation}, which means that all agents participate in every communication cycle, (ii) \textit{zero communication delay}, ensuring that belief states communicated among the agents are current, and (iii) \textit{exogenous switching mechanism}, where the harness architecture determines the communication topologies, not the agents themselves. 

We begin by describing an agent in terms of its belief state. Since we define consensus (the reconciliation of all belief states) as the ``goal" state of the MAS, this definition is appropriate for our framework.

\begin{definition}[Agent]\label{def:agent}
In a team of $N \in \mathbb{N}$ agents, each \emph{agent} is a reasoning unit
$
a_i$ where 
$i\in\{1,\ldots,N\},$
equipped with an internal belief state
$
x_i(k)\in\mathbb R^d
$
at communication round $k\in\mathbb Z_{\ge0}$. We index the discrete time evolution of the system by $k$, which updates every time the agents communicate according to a given topology. Communication among agents includes sharing information, registering disagreement, and proposing modifications.
The vector $x_i(k)$ represents the current $d$-dimensional assessment,
proposal, preference vector, or structured reasoning state of agent
$a_i$, for some $d \in \mathbb{N}$.

The collective state of the multi-agent system is the matrix
\[
\mathbf{X}_k
=
\begin{bmatrix}
x_1(k)^\top\\
x_2(k)^\top\\
\vdots\\
x_N(k)^\top
\end{bmatrix}
\in\mathbb R^{N\times d}.
\]
\end{definition}

While $\mathbf{X}_k$ denotes an internal summary of the MAS, it does not contain information about the external organizing architecture of the harness. We describe subsystems of this architecture as communication topologies, which we model as communication graphs. 

\begin{definition}[Communication Graph]\label{def:commgraph}
A \emph{communication graph} is a directed graph
$
\mathcal G_\sigma=(V,E_\sigma),
$ where 
$
\sigma\in\Sigma:=\{1,\ldots,m\} \text{ for } m \in \mathbb{N}.
$
The vertices 
$
V=\{a_1,\ldots,a_N\}
$
are the set of agents and edges $E_\sigma\subseteq V\times V$ are the set of
communication links active under mode $\sigma$. $\Sigma$ is the set of topologies, where each topology patterns a mode of communication.
Each graph $\mathcal G_\sigma$ is encoded by a nonnegative weight matrix
$
W_\sigma=[w_{ij}^{(\sigma)}]\in\mathbb R^{N\times N},
$
where $w_{ij}^{(\sigma)}>0$ means that agent $a_i$ incorporates
information from agent $a_j$ under mode $\sigma$. In graphical terms, the $ij$-th entry in $W_\sigma$ determines the edge between vertex $a_i$ and vertex $a_j.$

\begin{definition}[Doubly Stochastic]\label{def:doubly stochastic}
A real matrix $M$ is doubly stochastic if all its entries are nonnegative, and the sum of the entries in every row and every column is exactly 1.
\end{definition}

We assume that each $W_\sigma$ is doubly stochastic. Hence, elementary properties of matrix multiplication guarantee that
\[
W_\sigma\mathbf{1}=\mathbf{1},
\qquad
 W_\sigma^\top \mathbf{1}=\mathbf 1,
\]
where $\mathbf 1\in\mathbb R^N$ is the homogeneous vector of all ones.
\end{definition}
Now we formally define consensus, the central coordination objective of the harness. 

\begin{definition}[Consensus]\label{def:consensus}
The agents are said to be in \emph{consensus} at round $k$ if there
exists a vector $\bar x_k\in\mathbb R^d$ such that
$
x_1(k)=x_2(k)=\cdots=x_N(k)=\bar x_k .
$
Equivalently,
$
\mathbf{X}_k=\mathbf 1\,\bar x_k^\top .
$ The set
\[
\mathcal C
=
\left\{
\mathbf{X}\in\mathbb R^{N\times d}:
\mathbf{X}=\mathbf 1 c^\top,\; c\in\mathbb R^d
\right\}
\]
is called the \emph{consensus manifold}. The consensus manifold describes the set of all possible MAS aggregate belief matrices where each of the rows, representing individual agentic belief states, are the same.
\end{definition} 

By this definition, the agents have reached consensus when their individual belief vectors $x_i(k)$ are all equivalent. Until the MAS reaches consensus, the aggregate disagreement level of the system is $> 0$. 
\begin{definition}[Disagreement]\label{def:disagreement}
We define disagreement as the collective difference between each agent's belief state and the \textit{average} of all the agents' belief states. The system state $\mathbf{X}_k$, therefore, can be decomposed as:
$$
\mathbf{X}_k = Average + Disagreement.
$$
Mathematically, we write this as 
\begin{equation}
\mathbf{X}_k = \frac{1}{N}\textbf{1}\textbf{1}^\top \mathbf{X}_k+ P\mathbf{X}_k
\end{equation}
where $P\mathbf{X}_k$ is the \textit{disagreement component} of $\mathbf{X}_k.$
Notice that $\textbf{1}\textbf{1}^\top$ yields a $N \times N$ matrix where every entry is 1. Multiplying this matrix by $\mathbf{X}_k$ and then dividing every entry by $N$ results in homogeneous column vectors $v_1,...v_N$, where each entry in $v_j$ corresponds to the average of all entries in column $j$ of $\mathbf{X}_k.$
We define the disagreement subspace
\begin{equation}
\mathbf 1^\perp
=
\{v\in\mathbb R^N:\mathbf 1^\top v=0\}.
\end{equation}
This subspace consists of vectors whose entries sum to 0. Because the sum of deviations from a mean must be 0, the disagreement matrix
\begin{equation}
P
=
I-\frac1N\mathbf 1\mathbf 1^\top
\end{equation}
is the orthogonal projection matrix onto the disagreement subspace.

The aggregate disagreement level measures the total disagreement among all the agents in the MAS. We define it as
\begin{equation} 
D(k)
=
\frac1{\sqrt N}\|P\mathbf{X}_k\|_F .
\label{eq:disagreement}
\end{equation}
The Frobenius norm quantifies the size of the disagreement component, and we normalize by $\frac{1}{\sqrt{N}}$ to prevent the number of agents $N$ from distorting the disagreement level. 
Then
$
D(k)=0
$
if and only if $\mathbf{X}_k\in\mathcal C$, that is, if and only if the agents
are in consensus.
\end{definition}

The harness drives the aggregate disagreement level beneath a tolerance $\varepsilon > 0$ through communication between the agents. Differences in communication topologies effect the speed at which consensus is reached. When each agent can communicate with every other agent, consensus can be reached faster. The graph designating this topology would be a connected graph $K_N$, as shown in Fig. \ref{fig:fig2}. Clearly, in a topology where none of the agents communicated with each other, the aggregate disagreement level would remain the same, and the system would not reach consensus. We call the rate at which a particular communication topology changes the aggregate disagreement level the consensus-contraction rate.

\begin{definition}[Consensus-Contraction Rate]\label{def:contraction}
For a communication mode $\sigma\in\Sigma$, the
\emph{consensus-contraction rate} $\lambda_\sigma \in \mathbb{R}$ is defined as
\[
\lambda_\sigma
:=
\left\|
W_\sigma\big|_{\mathbf 1^\perp}
\right\|_2 .
\]
Equivalently, $\lambda_\sigma$ is the smallest constant satisfying
\[
\|W_\sigma v\|_2
\le
\lambda_\sigma\|v\|_2,
\qquad
\forall v\in\mathbf 1^\perp .
\]

The communication mode $\sigma$ is called \emph{consensus contracting}
if
$
\lambda_\sigma<1.
$
In this case, one communication round under mode $\sigma$ strictly
contracts the disagreement component.
\end{definition}

 Now we model how the collective belief state $\mathbf{X}_k$ (Definition \ref{def:agent}) evolves at each communication round $k.$ We capture the belief dynamics via a switched affine map:

\begin{equation}
\mathbf{X}_{k+1}
=
W_{\sigma_k}\mathbf{X}_k,
\qquad
k\in\mathbb Z_{\ge0},
\label{eq:consensus}
\end{equation}
where $\sigma_k\in\Sigma$ is the active communication mode.

\begin{remark}[Interpretation as Agentic Deliberation]
The switched system \eqref{eq:consensus} provides an abstract model of
a multi-agent reasoning process. At each communication round, the
active graph $\mathcal G_{\sigma_k}$ determines which agents exchange
information and how strongly neighboring opinions influence one
another through the weight matrix $W_{\sigma_k}$.
Because $W_{\sigma_k}$ is doubly stochastic, each update preserves the
collective average belief while reducing disagreement among agents.
The switching signal $\sigma_k$ therefore acts as a communication
policy that dynamically selects the interaction topology according to
the current level of disagreement. In the context of agentic AI,
different modes may correspond to sparse communication among a small
set of agents or dense deliberation involving the entire team,
thereby trading communication cost against consensus speed.
\end{remark}

\begin{remark}[Invariance of $W_\sigma$]
Since each matrix $W_\sigma$ is doubly stochastic,
every row of \eqref{eq:consensus} is a weighted average of the beliefs
received from neighboring agents in the communication graph
$\mathcal G_{\sigma_k}$. Furthermore, the disagreement subspace $\mathbf 1^\perp$ is invariant with respect to $W_\sigma.$ We can see this as follows. The disagreement subspace $\mathbf 1^\perp$ is the orthogonal complement space to the agreement subspace, which is $\text{span}(\mathbf 1)$. Thus, $\forall v \in \mathbf 1^\perp, w^\top v = 0$, where $w \in \text{span}(\mathbf 1)$. But $w = c \mathbf 1$ for some constant scalar $c \in \mathbb{R}.$ Therefore, $c \mathbf 1^\top (W_\sigma v)=(c\mathbf 1^\top W_\sigma)v = c(W_\sigma^\top \mathbf 1)^\top v = c\mathbf 1^\top v = 0.$ Thus, $W_\sigma$ maps vectors inside $\mathbf 1^\perp$ back into $\mathbf 1^\perp.$ Note that we already know that $\text{span}(\mathbf 1)$ is invariant with respect to $W_\sigma$ because the doubly stochastic property guarantees that $W_\sigma \mathbf 1 = \mathbf 1.$ 

By \eqref{eq:disagreement} the disagreement at $k+1$ satisfies
$$
D(k+1) = \frac{1}{\sqrt{N}} \norm{P \mathbf{X}_{k+1}}_F.
$$
Substituting \eqref{eq:consensus}, we see
$$
D(k+1) = \frac{1}{\sqrt{N}}\norm{PW_{\sigma_k}\mathbf{X}_k}_F.
$$
The matrix $P$ projects $\mathbf{X}_k$ onto the disagreement subspace $\mathbf 1^\perp.$ But because $\mathbf 1^\perp$ and $\text{span}(\mathbf 1
)$ are invariant with respect to $W_{\sigma_k}$, we can commute $P$ and $W_{\sigma_k}$:
$$
D(k+1) = \frac{1}{\sqrt{N}}\norm{W_{\sigma_k}P\mathbf{X}_k}_F.
$$
Now we can directly apply Def. \ref{def:contraction}, because $P\mathbf{X}_k \in \mathbf 1^\perp.$ 
$$
D(k+1) \leq \frac{1}{\sqrt{N}} \lambda_{\sigma_k} \norm{P \mathbf{X}_k}_F.
$$
So by \eqref{eq:disagreement} we obtain
\begin{equation} \label{eq:contracting}
D(k+1)
\le
\lambda_{\sigma_k}D(k),
\end{equation}
where one communication round modifies the distance to the consensus
manifold by at least the factor $\lambda_{\sigma_k}$. Note that $\lambda_\sigma$ shrinks as the disagreement reduction potency of mode $\sigma$ grows. 
\end{remark}

\begin{remark}[Beyond Consensus-Contracting Modes]
The assumption $\lambda_\sigma<1$ for every communication mode is
sufficient but not necessary for consensus. More generally, consensus
may still occur when some modes are non-contracting or even locally
expansive, provided that the overall switching sequence produces a net
reduction of disagreement. Such situations can be analyzed using
average-contraction or switched-system stability conditions. The
present work focuses on the simpler case in which each communication
mode is individually consensus contracting, enabling explicit
expressions for the consensus time and certified budget.
\end{remark}

A core principle of our framework is the incorporation of resource-management into the control objective. We begin by defining the budget and cost variables. 

\begin{definition}[Token Budget Dynamics]\label{def:budget}
Let
$
b_k\in\mathbb R_{\ge0}
$
denote the remaining token budget at communication round $k$.
Each communication mode
$\sigma\in\Sigma$
incurs a token cost
$c_\sigma>0$.
The budget evolves according to
\begin{equation}
b_{k+1}
=
b_k-c_{\sigma_k},
\qquad
b_0>0.
\label{eq:budget}
\end{equation}
\end{definition}

As illustrated in Figure 2, we optimize between speed and cost through norm-threshold switching.

\begin{definition}[Norm-Threshold Switching]\label{def:switch}
Let $\eta>0$ be a prescribed disagreement threshold.
The active communication mode is selected according to
\begin{equation}
\sigma_k=
\begin{cases}
1, & D(k)>\eta,\\[2mm]
2, & D(k)\le \eta.
\end{cases}
\label{eq:rule}
\end{equation}

Mode 2 is activated whenever the disagreement falls beneath the threshold $\eta$,
while Mode 1 is activated as long the disagreement exceeds it.
\end{definition}

The aggregate belief state $\mathbf{X}_k$ of the MAS does not include budget dynamics. This deficiency motivates an augmented state description of the MAS system. 

\begin{definition}[Augmented State]
The augmented state of the multi-agent system is

\begin{equation}
z_k=(\mathbf{X}_k,b_k)
\in
\mathbb R^{N\times d}\times\mathbb R_{\ge0}.
\end{equation}

The first component represents collective beliefs,
while the second component represents the remaining token budget.
\end{definition}

Now we mathematically formulate what it means for a MAS to be ``safe." The $\varepsilon$-Consensus Safe Set unifies all the previous aspects of the framework into a rigorous characterization of successful goal states. This definition also lays theoretical foundations for the discussion of boundary conditions and certificate regions that follows in the next section. 
\begin{definition}[$\varepsilon$-Consensus Safe Set]
\label{def:epssafe}

Let $\varepsilon>0$ denote an acceptable consensus tolerance.
The \emph{$\varepsilon$-consensus safe set} is
\begin{equation} \label{eq:omega_ep}
\Omega_\varepsilon
=
\left\{
(\mathbf{X},b)\in
\mathbb R^{N\times d}\times\mathbb R_{\ge0}
:
D(\mathbf{X})\le\varepsilon,
\;
b\ge0
\right\}.
\end{equation} 

A state $(\mathbf{X},b)\in\Omega_\varepsilon$ satisfies two properties:

\begin{enumerate}
\item The collective disagreement among agents is bounded by
      $\varepsilon$, that is, $D(\mathbf{X})\le\varepsilon$,
      so the agent network has reached practical consensus.
\item The remaining token budget is nonnegative, $b\ge0$,
      ensuring that no computational resource constraint has been violated.
\end{enumerate}

Consequently, \eqref{eq:omega_ep} represents the set of states for
which the collaborative reasoning task is considered successfully
completed while remaining within the available communication budget.
\end{definition}

\begin{remark}[Interpretation for Multi-Agent AI Systems] The parameter $\varepsilon$ designates a tolerance for the degree of agreement required
before the system terminates deliberation.
For example, in a network of LLM-based agents performing software
architecture planning, each row of $\mathbf{X}_k$ may encode an agent's
preferences regarding deployment strategy, database selection,
authentication mechanisms, scalability, cost, and observability.
The condition
$ D(\mathbf{X}_k)\le\varepsilon $
implies that all agents have converged to sufficiently similar
proposals, even if their belief vectors are not exactly identical.

Similarly, in scientific discovery, collaborative theorem proving,
multi-agent code generation, or autonomous planning, exact consensus
is rarely necessary. Instead, practical agreement within a tolerance
$\varepsilon$ is sufficient for selecting a final decision and
terminating further communication.
The budget constraint $b_k\ge0$ reflects finite computational resources,
such as token limits, API costs, inference credits, wall-clock
constraints, or energy consumption. Therefore,
$\Omega_\varepsilon$ characterizes successful completion of a
multi-agent reasoning task with both consensus and resource
feasibility guarantees.
\end{remark}

\section{Consensus and Budget Certificates}
We derive boundary conditions for a consensus-budget certificate region and prove its forward invariance. We assume for the sake of simplicity that there are two communication arrangements, as in Fig. \ref{fig:fig2}. We express these arrangements as two switching modes, $\sigma = 1$ and $\sigma = 2$. We assume that the consensus-contraction rates $\lambda_{1}, \lambda_2$ associated with $\mathcal{G}_{1}$ and $\mathcal{G}_{2}$ are both less than one.

By \eqref{eq:contracting}, our assumption that $\lambda_\sigma < 1$ implies that the aggregate disagreement is strictly decreasing under every active
communication mode. Because $\lambda_{\sigma_k}$ is a \textit{rate}, we can formulate exact expressions for the number of rounds in which we can reach consensus from an initial disagreement level $D_0$. Moreover, given the computational cost $c_\sigma$ of each communication topology $\sigma$ for a given round, we can predict the total cost. 

\begin{theorem}[Resource-Aware Consensus]
\label{thm:rac}

Assume $ D(0)>\eta>\varepsilon>0.$ Define
\begin{equation}
K_1
:=
\left\lceil
\frac{\log(\eta/D(0))}
{\log\lambda_1}
\right\rceil,
\qquad
K_2
:=
\left\lceil
\frac{\log(\varepsilon/\eta)}
{\log\lambda_2}
\right\rceil,
\label{eq:Ks}
\end{equation}
for consensus-contraction rates $0< \lambda_1, \lambda_2 < 1 $,
and let $K^* = K_1+K_2$.

Then:

\begin{enumerate}
\item
During the high-disagreement phase,
\[
D(k)
\le
\lambda_1^kD(0),
\qquad
0< k < K_1.
\]

\item
During the low-disagreement phase,
\[
D(k)
\le
\lambda_2^{\,k-K_1}\eta,
\qquad
k\geq K_1>0.
\]

\item
The trajectory reaches
$\varepsilon$-consensus in at most $K^*$ rounds:
\[
D(K^*)
\le
\varepsilon.
\]

\item
The corresponding communication cost is
\[
B^*
=
K_1c_1+K_2c_2,
\] where $c_1,c_2$ are the per-round token costs of modes 1 and 2 respectively.
\end{enumerate}

\end{theorem}

\begin{proof}
To derive the expression for $K_1$, observe that repeated application of \eqref{eq:contracting} yields $D(K_1) \leq \lambda_1^{K_1}D(0)$ for $K_1 > 0.$
The switching rule activates mode $1$ while
$D(k)>\eta$ and mode $2$ otherwise.
$K_1$ is an upper bound on the number of rounds until $D(k) \leq \eta$. To solve for $K_1$, set $\eta = \lambda_1^{K_1}D(0) \implies \lambda_1^{K_1} = \frac{\eta}{D(0)} \implies K_1 = \log_{\lambda_1}\Big(\frac{\eta}{D(0)} \Big)$. Changing base, $K_1 = \frac{\log (\eta / D(0))}{\log \lambda_1}$. Since we are counting discrete communication rounds, we must have $K_1 \in \mathbb{Z}_{\geq 0}$. We also want to guarantee $D(K_1) \leq \eta $. Thus, we apply the ceiling function $K_1
=
\left\lceil
\frac{\log(\eta/D(0))}
{\log\lambda_1}
\right\rceil.$
The expression for $K_2$ is derived similarly. Due to the construction of $K_1,K_2$, repeated application of \eqref{eq:contracting} yields the bounds in (1) and (2).  The definition of $K_1$ implies
$ \lambda_1^{K_1}D(0)\le\eta$, while the definition of $K_2$ implies $ \lambda_2^{K_2}\eta\le\varepsilon$.
Combining both estimates yields $D(K^*)\le\varepsilon$.

Finally, the total token expenditure equals the number of rounds
spent in each mode multiplied by the corresponding communication cost,
giving $B^*=K_1c_1+K_2c_2$.

\end{proof}

\begin{corollary}[Topology Trade-Off]
\label{cor:tradeoff}

Define
\[
K_{\mathrm{cmpl}}
=
\left\lceil
\frac{\log(\varepsilon/D(0))}
{\log\lambda_1}
\right\rceil,
K_{\mathrm{ring}}
=
\left\lceil
\frac{\log(\varepsilon/D(0))}
{\log\lambda_2}
\right\rceil.
\qquad
\]

Then if $\lambda_1 < \lambda_2,$
\[
K_{\mathrm{cmpl}}
\le
K^*
\le
K_{\mathrm{ring}}.
\]

Moreover, if
\[
\frac{c_1}{c_2}
\ge
\frac{\log\lambda_1}{\log\lambda_2},
\]
then
\[
K_{\mathrm{ring}}c_2
\le
B^*
\le
K_{\mathrm{cmpl}}c_1.
\]

Therefore, the adaptive communication policy achieves a compromise
between communication speed and token expenditure: it is faster than
a ring-only strategy and cheaper than a complete-graph strategy. 
\end{corollary}

\begin{definition}[Consensus-Budget Certificate Region]
\label{def:cert}

Let
\[
B^*(D)
=
\left\lceil
\frac{\log(\eta/D)}
{\log\lambda_1}
\right\rceil c_1
+
\left\lceil
\frac{\log(\varepsilon/\eta)}
{\log\lambda_2}
\right\rceil c_2.
\]

The set
\[
\Omega_{\mathrm{cert}}
=
\left\{
(\mathbf{X},b)
\in
\mathbb R^{N\times d}\times\mathbb R_{\ge0}
:
b\ge B^*(D(\mathbf{X}))
\right\}
\]
is called the consensus-budget certificate region.
\end{definition}

\begin{theorem}[Consensus-Budget Certificate]
\label{thm:cert}

If
\[
(\mathbf{X}_0,b_0)\in\Omega_{\mathrm{cert}},
\]
then the corresponding trajectory reaches the
$\varepsilon$-consensus safe set $\Omega_\epsilon$ as in \eqref{eq:omega_ep} in at most $K^*$ communication rounds. Equivalently, $D(\mathbf{X}_{K^*})\le\varepsilon $ and $ b_{K^*}\ge0$.
\end{theorem}

\begin{proof} Membership in $\Omega_{\mathrm{cert}}$ implies $b_0 \ge B^*(D(\mathbf{X}_0)).$ By Theorem~\ref{thm:rac}, the trajectory reaches $
D(\mathbf{X}_{K^*})\le\varepsilon $ after at most $K^*$ rounds and consumes at most $B^*(D(\mathbf{X}_0))$ tokens. Therefore $ b_{K^*}
= b_0-B^*(D(\mathbf{X}_0)) \ge0$, which yields $(\mathbf{X}_{K^*},b_{K^*})
\in
\Omega_\varepsilon.
$
\end{proof}

\begin{remark}[Operational Interpretation]

The boundary of the certificate region is given by $b=B^* (D(\mathbf{X})).$ States above this surface possess sufficient communication budget to guarantee $\varepsilon$-consensus, while states below it do not admit such a certification.

Consequently, before launching a multi-agent reasoning workflow,
an operator can compute the initial disagreement $D(\mathbf{X}_0)$ and verify 
$
b_0\ge B^*(D(\mathbf{X}_0)).
$

If the inequality holds, finite-time consensus is guaranteed before
the available token budget is exhausted.
\end{remark}

\begin{remark} [Calculating $\lambda_\sigma$]
In the numerical results section, we will calculate $\lambda_\sigma$ for weight matrices $W_\sigma$ which we design to be symmetric and circulant. Given these two properties, it is easy to calculate the consensus contraction rates. Because $W_\sigma$ is symmetric, the value for $\lambda_\sigma$ is the largest eigenvalue of $W_\sigma.$ Because $W_\sigma$ is circulant, we can calculate its eigenvalues by computing the Discrete Fourier Transform of the first row. 

For an $N \times N$ symmetric matrix $W$, its eigenvectors form an orthonormal basis for $\mathbb{R}^N.$ Let $\{v_1, v_2,...,v_N\}$ be the orthogonal basis of normalized eigenvectors, with corresponding eigenvalues $\{\lambda_1, \lambda_2,..., \lambda_N\}$ ordered such that $|\lambda_1| \geq |\lambda_2| \geq ...\geq|\lambda_N|.$ Take any vector $v \in \mathbf 1^\perp$. Since the disagreement space $\mathbf 1^\perp$ is a subspace of $\mathbb{R}^N$, $v = a_1v_1+a_2v_2+...a_nv_n$ for some $a_1,a_2,...,a_N \in \mathbb{R}.$ 

We defined $\lambda = \norm{W|_{1^\perp}}_2$. By the definition of the spectral norm, 
$$
\lambda = \max_{v \neq 0} \frac{\norm{Wv}_2}{\norm{v}_2} \quad \forall v \in \mathbf 1^\perp.
$$
Note that $Wv_i = \lambda_iv_i$ since $v_i$ is an eigenvector of $W.$ Then,
\begin{align*}
\norm{Wv}_2 & = \norm{W(a_1v_1+a_2v_2+\dotsi+a_Nv_N)}_2 \\
& = \norm{a_1 \lambda_1 v_1 + a_2\lambda_2v_2+\dotsi+a_n\lambda_Nv_N}_2 \\
& = \sqrt{\big(a_1 \lambda_1 v_1 + a_2\lambda_2v_2+\dotsi+a_n\lambda_Nv_N\big)^2} \\
& = \sqrt{a_1^2\lambda_1^2+\dotsi a_N^2\lambda_N^2} \\
&\leq \sqrt{\lambda_1^2(a_1^2+\dotsi + a_N^2)} \\
& = \lambda_1 \sqrt{(a_1^2 + \dotsi + a_N^2)} = \lambda_1 \norm{v}_2 \\
\implies & \norm{Wv}_2 \leq \lambda_1 \norm{v}_2.
\end{align*}
Thus $\lambda = \lambda_1.$ 

We apply the Discrete Fourier Transform to the first row of our circulant matrix $W$ to obtain the eigenvalues. Recall that the rows of a circulant matrix preserve the ordering of a sequence of elements $w_0,w_1,w_2,\dots,w_{N-1}$ while starting at a different element in each row. For an $N\times N$ circulant matrix W and an $N$-dimensional vector $x$, we define 
$$
y = Wx = \begin{bmatrix}
    w_0 & w_1 & w_2 & \dotsi & w_{N-1} \\
    w_{N-1} & w_0 & w_1 & w_2 & \dotsi \\
    \ddots & \ddots & \ddots & \ddots &\ddots \\
    w_1 & w_2 & \dots & w_{N-1} & w_0
\end{bmatrix} \begin{bmatrix}
    x_0 \\ x_1 \\ \vdots \\ x_{N-1}
\end{bmatrix}.
$$
Then we can compactly represent the $l$th element of the vector $y$ as $y_l = \sum_{j = 0}^{N-1}w_{j-l}x_j$, where we count $j - l \mod{n}$. For every circulant matrix, the $k$th eigenvector $x^{(k)}$ for $k = 0, 1, \dots, N-1$ is
$$
x^{(k)} = \begin{bmatrix}
    \omega_N^{0k} \\ \omega_N^{1k}\\ \omega_N^{2k} \\ \vdots \\ \omega_N^{(N-1)k}
\end{bmatrix}
$$
where $\omega_N = e^{{2\pi i}/N}$ is the primitive $N$th root of unity. If we combine these eigenvectors into a matrix, we get the matrix
$$
F = \begin{bmatrix}
    x^{(0)} & x^{(1)} & x^{(2)} & \cdots  & x^{(N-1)}
\end{bmatrix}
$$
where $F_{jk} = x_j^{(k)} = \omega_N^{jk}$. $Fx$ yields the Discrete Fourier Transform of $x.$ To see how we can use this to calculate our eigenvalues, consider the eigenvector $x^{(k)}$ and let $y^{(k)} = Wx^{(k)}$. Using the summation above, the $l$th component is 
$$y^{(k)}_l = \sum_{j = 0}^{N-1} w_{j-l}\omega_N^{jk} = \omega_N^{lk} \sum_{j = 0}^{N-1} w_{j-l}\omega_N^{(j-l)k}.$$
Since $\omega_N^{lk}$ is the $l$th component of $x^{(k)}$ and the summation is unchanged by starting location within $N-1$, 
$$
Wx^{(k)} = \lambda_k x^{(k)}
$$
where the eigenvalue $\lambda_k = \sum_{j = 0}^{N-1}w_j\omega_N^{jk}.$ Therefore, if the vector $\hat{w} = \begin{bmatrix}
    \lambda_0 & \lambda_1 & \lambda_2 & \dots  & \lambda_{N-1}
\end{bmatrix}$
lists all the eigenvalues of $W$, then 
$$\hat{w} = Fw$$ where $w$ is the first row of $W.$
\end{remark}

\section{Numerical Results}
The next two sections demonstrate how our control-theoretic framework provides mathematical handles for consensus-driven LLM-MAS. First we will illustrate the dynamics of our switched-system through simulations where the weighted-average update $\mathbf{X}_{k+1}=W_{\sigma_k}\mathbf{X}_k$ in \eqref{eq:consensus} is the
algebraic surrogate for LLM debate. Then, we will present a live LangGraph deployment with major AI platforms to show that the framework accurately captures the behavior of real MAS, agnostic to LLM choice.  

We use the same MAS configuration across simulations and LangGraph validations. We commission 5 agents to reach consensus on a plan for a cloud software architecture. Each agent assumes a distinct professional role whose proposals reflect specialist priorities, listed in Table \ref{tab:roles}. The biases of each role will not effect the simulations, but will be crucial later in determining the behavior of the LLMs.
\begin{table}[h]
\centering
\caption{Agent Roles}
\label{tab:roles}
\begin{tabular}{@{}lll@{}}
\toprule
Agent & Role & Dominant bias \\
\midrule
$a_1$ & Planner       & Balanced; timeline feasibility \\
$a_2$ & Architect     & Microservices, containers, scalability \\
$a_3$ & Sec.\ Reviewer & Authentication, compliance, tracing \\
$a_4$ & Cost Optimizer & Simpler stack, NoSQL, lower cost \\
$a_5$ & DevOps Eng.   & CI/CD, observability, K8s \\
\bottomrule
\end{tabular}
\end{table}
\noindent
Each agent holds a \emph{design-preference vector}
$x_i\in[0,1]^6$ encoding its position on six architectural
axes: (1)~service coupling, (2)~API style, (3)~database strategy,
(4)~deployment, (5)~authentication, (6)~observability. The design preference vector $x_i$ represents the belief state of agent $a_i$.
The state matrix $\mathbf{X}_k\in\mathbb{R}^{5\times6}$ stacks all agent vectors
row-wise.

The agents reach consensus through debating in one of two communication modes, where each mode is encoded by a weight matrix $W_\sigma$. The two modes, summarized in Table \ref{tab:weights}, correspond to the topologies in
Fig.~\ref{fig:fig2}. 
\begin{table}[h]
\centering
\caption{Weight Matrices}
\label{tab:weights}
\begin{tabular}{@{}lllc@{}}
\toprule
Mode & Graph & $W_\sigma$ & $c_\sigma$ \\
\midrule
$\sigma\!=\!1$ & Complete $K_5$ &
  $w_{ii}\!=\!0.4$,\;$w_{ij}\!=\!0.15$ ($i\!\ne\!j$) & $c_1 > c_2$ \\
$\sigma\!=\!2$ & Ring $C_5$ &
  $w_{ii}\!=\!0.6$,\;$w_{ij}\!=\!0.2$ ($j = i \pm 1$) & $c_2 < c_1$\\
\bottomrule
\end{tabular}
\end{table}
\noindent
We chose the values $w_{ij}$ semi-arbitrarily. Our framework assumes that the weight matrices are doubly stochastic. Furthermore, we gave the highest weight to the agent's own opinion to preserve the agent's original goals while compromising with the team. This is why $w_{ii} > w_{ij}$. We also designed $W_\sigma$ to be circulant and symmetric, so that we could easily calculate the consensus contraction rates $
\lambda_\sigma$. Apart from these two principles, the parameter values were chosen for arithmetic convenience. 
We can calculate the consensus-contraction rates for each weight matrix: 
$\lambda_2\approx0.724$ and $\lambda_1=0.25$.
Token costs $c_\sigma$ reflect the lower cost of ring systems exchanging compact per-neighbor
messages vs. the higher cost of complete-graph systems broadcasting full proposals to all peers.

\begin{algorithm}[H]
\caption{Norm-Threshold Adaptive Consensus}\label{alg:consensus}
\begin{algorithmic}
\STATE \textbf{Require:} $\mathbf{X}_0$, $b_0$, $\eta$, $\varepsilon$, $W_1,W_2$, $c_2<c_1$
\STATE \textbf{Ensure:} $\mathbf{X}_{K^*}$ (consensus) or \textsc{BudgetFail}
\STATE \hspace{0.5cm}$k\gets0$
\STATE \hspace{0.5cm}\textbf{while} {$D(k)>\varepsilon$} \textbf{do}
  \STATE \hspace{1cm}$\sigma_k\gets1$ if $D(k)>\eta$, else $\sigma_k\gets2$
  \STATE\hspace{1cm}\textbf{if} {$b_k<c_{\sigma_k}$} \textbf{then return} \textsc{BudgetFail}
  \STATE \hspace{1cm}\textbf{end if}
  \STATE \hspace{1cm}$\mathbf{X}_{k+1}\gets W_{\sigma_k}\mathbf{X}_k$;\;
         $b_{k+1}\gets b_k-c_{\sigma_k}$;\;
         $k\gets k+1$
\STATE \hspace{0.5cm}\textbf{end while}
\STATE \hspace{0.5cm}\textbf{return} $\mathbf{X}_k$
\end{algorithmic}
\end{algorithm}
Algorithm~\ref{alg:consensus} implements the switching policy. Communication rounds continue until $D(k) \leq \epsilon $ or the token budget $b_k $ is less than the cost of the next communication round $c_{\sigma_k}.$ In the LangGraph deployments, cost $c_{\sigma_k}$ was dynamic and impossible to predict before the communication round, and so we returned \textsc{BudgetFail} once $b_k < 0.$

\subsection{Illustrative Example of Algorithm~\ref{alg:consensus}}
For each agent $a_i$, we sample every entry in its initial belief state $x_0$ from a normal distribution centered at 0.5 with a standard deviation of 0.15, truncated at the endpoints 0 and 1. Each entry in $x_0$ denotes the agent's original position on one of the six architectural features. The MAS belief state $\mathbf{X}_0$ is shown in Table \ref{tab:init_belief}.
\begin{table}[h]
\caption{Initial belief state $\mathbf{X}_0$}
\label{tab:init_belief}
\begin{center}
\begin{tabular}{c c c c c c c} 
\toprule
 Agent / Feature & (1) & (2) & (3) & (4) & (5) & (6) \\  
\midrule
 $a_1$ & 0.57 & 0.54 & 0.57 & 0.13 & 0.55 & 0.52 \\ 
 $a_2$ & 0.39 & 0.22 & 0.57 & 0.09 & 0.74 & 0.57 \\
 $a_3$ & 0.56 & 0.31 & 0.46 & 0.58 & 0.48 & 0.75 \\
 $a_4$ & 0.79 & 0.44 & 0.40 & 0.28 & 0.64 & 0.56 \\
 $a_5$ & 0.33 & 0.63 & 0.37 & 0.49 & 0.53 & 0.66 \\ 
\bottomrule
\end{tabular}
\end{center}
\end{table}
The initial disagreement among the agents is $D(0)\approx0.30$. We set the switching threshold $\eta=0.1$, and the consensus tolerance 
$\varepsilon=0.03$, with an initial token budget $b_0=2000$. Because these simulations do not expend tokens, we set constant token costs based on the number of edges in a given communication graph: $c_2\!=\!100$ ($5\!\times\!2\!\times\!10$~tok) vs.  $c_1\!=\!600$  
($5\!\times\!4\!\times\!30$~tok), giving $c_1/c_2=6$.
Theorem~\ref{thm:rac} predicts $K_2=1$, $K_1=4$, $K^*=5$,
$B^*=1000$ as upper bounds.
The corollary condition $c_1/c_2=6\ge\log\lambda_1/\log\lambda_2\approx4.29$
holds, so Corollary~\ref{cor:tradeoff} applies.
Algorithm~\ref{alg:consensus} selects the topology at each round,
updates beliefs via \eqref{eq:consensus}, and decrements
the budget by $c_\sigma$.
The complete execution is traced in Table \ref{tab:exec}. 
\begin{table}[h]
\centering
\caption{Execution of Algorithm \ref{alg:consensus}}
\label{tab:exec}
\setlength{\tabcolsep}{3.5pt}
\small
\centering
\begin{tabular}{@{}crrrrp{3.0cm}@{}}
\toprule
$k$ & $D(k)$ & $\sigma_k$ & $c_{\sigma_k}$ & $b_k$ & Decision \\
\midrule
0 & 0.3009 & 2 & 600 & 2000
  & $D\!>\!\eta$: \emph{complete} $K_5$ \\[1pt]
1 & 0.0752 & 1 & 100 & 1400
  & $D\!\le\!\eta$: \emph{ring} $C_5$ \\[1pt]
2 & 0.0381 & 1 & 100 & 1300
  & $D\!>\!\varepsilon$: ring continues \\[1pt]
3 & \textbf{0.0254} & --- & --- & \textbf{1200}
  & $D\!\leq\varepsilon$: \textbf{consensus} \\
\bottomrule
\end{tabular}
\end{table}

\begin{figure*}[h]
    \centering
    \begin{subfigure} [t] {0.49\textwidth}
     \caption{}
        \centering
        \includegraphics[width=1.0\linewidth, trim=0.3cm 0.2cm 1.3cm 1.4cm, clip]{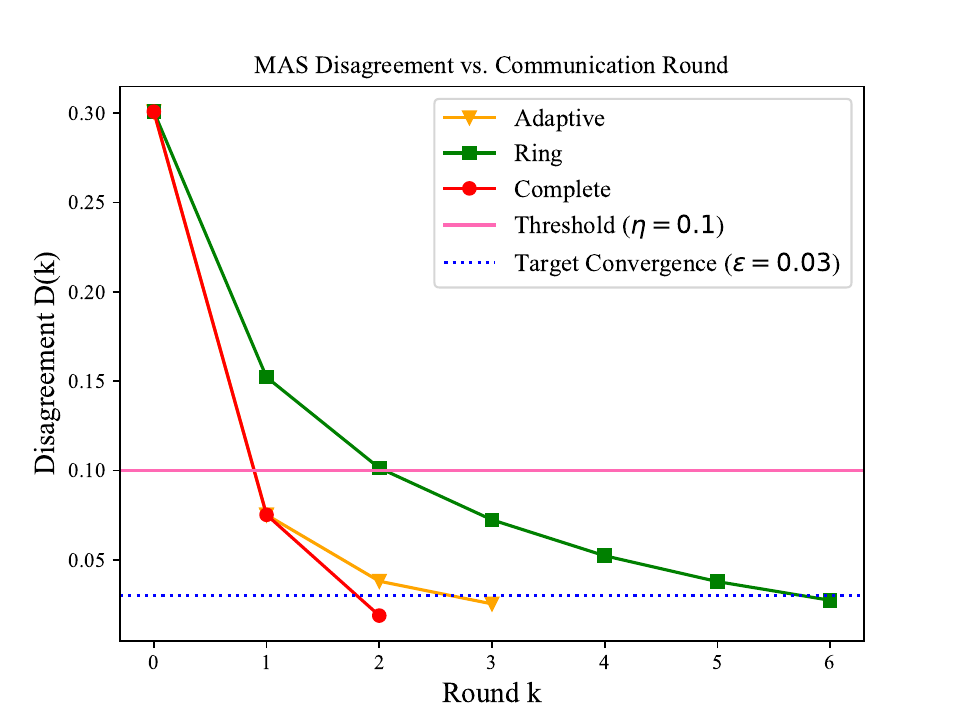}
         \label{fig:fig3}
    \end{subfigure}
     \begin{subfigure} [t] {0.49\textwidth}
      \caption{}
        \centering
         \includegraphics[width=1.0\linewidth, trim=0.3cm 0.2cm 1.3cm 1.4cm, clip]{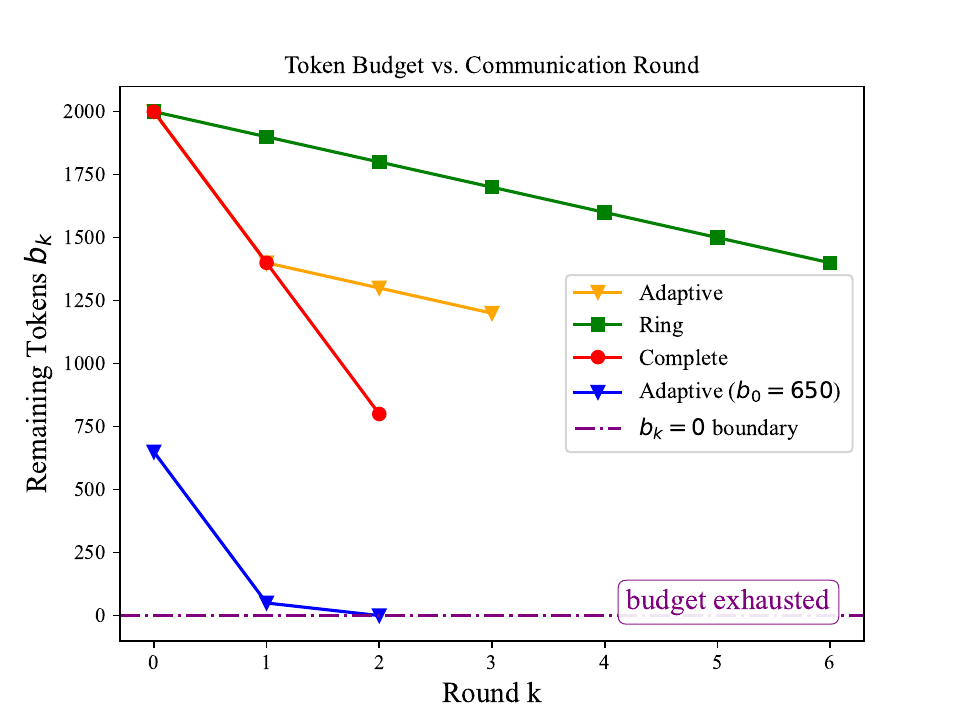}
       \label{fig:fig4}
    \end{subfigure}
    \caption{Panel (a) presents the disagreement dynamics across the three well-funded strategies. Panel (b) compares the budget dynamics between the well-funded strategies and the tight-budget run.}
    \label{fig:main_three_plots}
\end{figure*}

Theorem~\ref{thm:rac} yields theoretical verification of the execution. The spectral bound of Theorem~\ref{thm:rac} gives:
$D(1)\le\lambda_1 D(0)=0.25\times0.3009=0.0752$
(achieved exactly, since $W_2$ has eigenvalue $\lambda_1\!=\!0.25$ with
multiplicity $N\!-\!1$ on $\mathbf{1}^\perp$);
$D(3)\le\lambda_1\lambda_2^2 D(0)=0.25\times0.724^2\times0.3009=0.0394$
(actual $0.0254<0.0394$, faster due to energy in the faster eigenspace
$\lambda\!=\!0.276$ of $W_1$).
Since the final budget $b_3\!=\!1200\ge0$ and $D(3) \leq \epsilon,$ we have $(\mathbf{X}_3,b_3)\in\Omega_\varepsilon$. Since $B^*\!=\!1000\le b_0\!=\!2000$, the success of this execution confirms the guarantees given by Theorem~\ref{thm:cert}.

\subsection{Illustrative Example of Budget-Failure Run}
We consider the MAS path to consensus under three communication strategies.
\emph{Strategy~1 (adaptive)}: norm-threshold switching, $b_0=2000$. This strategy follows the switching policy given in the framework, where agents communicate globally under Mode 1 while $D(k) > \eta,$ and then communicate with their immediate neighbors under Mode 2. We expect this strategy to optimize between speed and cost.
\emph{Strategy~2 (complete-only)}: Mode~2 throughout, $b_0=2000$. Agents communicate globally at every round, reaching consensus faster but spending more tokens. 
\emph{Strategy~3 (ring-only)}: Mode~1 throughout, $b_0=2000$. Agents only communicate with their two neighbors, saving tokens but increasing the number of rounds required for consensus. All three strategies begin with an initial budget of 2000 tokens.

\begin{table}[h]
\centering
\caption{Execution summary. Strategies~1--3 reach $\Omega_\varepsilon$ ($b_\text{final}\ge0$).
Tight-budget run exhausts budget before $\varepsilon$-consensus,
confirming Theorem~\ref{thm:rac}(iv).}
\label{tab:summary}
\setlength{\tabcolsep}{4pt}
\begin{tabular}{lcccc}
\toprule
Strategy & Rounds & Tokens & $D_\text{final}$ & In $\Omega_\varepsilon$? \\
\midrule
1: Adaptive      & 3 & 800  & 0.0254 & \checkmark \\
2: Complete-only & 2 & 1200 & 0.0188 & \checkmark \\
3: Ring-only     & 6 & 600  & 0.0274 & \checkmark \\
Tight ($b_0\!=\!650$) & \textemdash & 600 & 0.0752 & $\times$ \\
\bottomrule
\end{tabular}
\end{table}

We also show how the MAS exhibits unsafe behavior when the initial budget falls outside of the bounds given by Definition~\ref{def:cert}.
\emph{Tight-budget run}: adaptive with $b_0=650<B^*=1000$. Because the initial token budget is less than the $B^*$ given by Theorem \ref{thm:rac}, we expect this strategy to trigger \textsc{Budget Fail}. We summarize the results in Table \ref{tab:summary}.
We ran this simulation with the same initial belief vector $\mathbf{X}_0$ given in Table \ref{tab:init_belief}, so $D(0), B^*$, etc. 
We plot the disagreement trajectories in Fig. \ref{fig:fig3}.
Complete (Strategy~2) converges fastest: $D(k)=0.25^k\cdot0.3009$
reaches $0.0188\le\varepsilon$ in 2~rounds. We would expect this because the matrix $W_{complete}$ averages the belief states of all the agents in the system at each communication round. The matrix $W_{ring}$ only averages adjoining belief vectors, so ring (Strategy~3) requires 6~rounds.
Adaptive (Strategy~1) exhibits the characteristic two-slope
trajectory: one steep drop under Mode~1 ($D$ falls from $0.3009$ to
$0.0752$ at round~1), then three gentler Mode-2 steps to $D=0.0254$---3~rounds
total, beating the worst-case bound of $K^*=5$.

We present the budget dynamics for the three well-funded strategies and the tight-budget run in Fig.~\ref{fig:fig4}. Strategies~1--3 remain strictly within the safe region ($b_k\ge0$), confirming forward invariance (Theorem~\ref{thm:cert}).
The kink in the adaptive curve at $k=1$ marks the Mode~$1 \to 2$
switch, where the per-round cost drops from 600 to 100~tokens.
The tight-budget curve exits $\Omega_\varepsilon$ during the Mode-2 step:
$b_0=650>c_1=600$ succeeds at $k=0$ ($b_1=50$) but then
$b_1=50<c_2=100$, triggering \textsc{BudgetFail} at $k=2$.

This simulation illustrates the chain in Corollary \ref{cor:tradeoff}:
$K_\text{cmpl}=2\le K^*=5\le K_\text{ring}=8$ and
$B_\text{ring}=800\le B^*=1000\le B_\text{cmpl}=1200$. The data in Table \ref{tab:summary} fits the bounds from the Corollary.
Adaptive saves 3~rounds vs.\ ring-only (at a cost of 200~extra tokens)
and saves 400~tokens vs.\ complete-only (at a cost of 1~extra round).
This Pareto efficiency is the core practical message: when both
latency and token cost carry weight, the adaptive policy dominates both
single-topology alternatives.

\subsection{AI Platform Implementation}
We deploy three popular LLM models, GPT-5.4-Mini, Gemini-2.5-Flash, and Claude Opus-4.8, on a LangGraph implementation of our framework. We used the benchmarks in Table \ref{tab:model_comparison} as a guideline for selecting these models. Due to the high token costs and longer compute times of GPT-5.5 and Gemini 3.1-Pro, we selected GPT-5.4-Mini and Gemini 2.5-Flash as capable representatives of their respective model families.

LangGraph is a low-level orchestration framework for building MAS workflows. In our scenario, the LangGraph architecture is the harness for the LLM agents. 
LangGraph represents a multi-agent workflow as a
\emph{StateGraph}: nodes transform a shared state object, and
conditional edges route execution based on state values.
The shared \texttt{ArchState} TypedDict encodes the complete
system state $(\mathbf{X}_k,b_k,D(k),\sigma_k)$ of
Definitions~\ref{def:agent}--\ref{def:switch}.
Five specialized agents---Planner, Architect, Security Reviewer,
Cost Optimizer, and DevOps Engineer---are LangGraph nodes that
read their neighbors' current proposals and output an updated
$d$-dimensional design vector as a text string. 
The key insight is that the weighted-average aggregation step
in \eqref{eq:consensus} can be applied by any LLM that
implements a doubly stochastic blending of its neighbors'
proposals, so all theoretical guarantees of
Theorems~\ref{thm:rac}--\ref{thm:cert} hold independently of
which specific model is deployed at each node.



Table~\ref{tab:mapping} summarizes the theory-to-code correspondence.
When live LLM agents are substituted for the matrix-multiply step \eqref{eq:consensus},
the LangGraph implementation inherits all guarantees of
Theorems~\ref{thm:rac}--\ref{thm:cert}:
forward invariance of $\Omega_\varepsilon$ is enforced by the
\texttt{should\_continue} guard; consensus time $K^*$ and budget
sufficiency $B^*$ are certified offline from $\lambda_\sigma$,
$c_\sigma$, and $D(0)$ before any inference call is made. It is important to clarify that we did not calculate $K^*, B^*$ in our deployments, because the parameters $\lambda_1,\lambda_2, c_1, c_2$ depend on hidden design features of the LLM. 

\begin{table}[h]
\centering
\caption{Theory-to-LangGraph correspondence.}
\label{tab:mapping}
\setlength{\tabcolsep}{4pt}
\footnotesize
\begin{tabular}{ll}
\toprule
Mathematical object & LangGraph component \\
\midrule
$\mathbf{X}_k\in\mathbb{R}^{N\times d}$       & \texttt{proposals} in \texttt{ArchState} \\
$b_k$                          & \texttt{budget} in \texttt{ArchState} \\
$D(k)=\|P\mathbf{X}_k\|_F/\!\sqrt{N}$& \texttt{D} in \texttt{ArchState} \\
$\sigma_k$ (Def.~\ref{def:switch})& \texttt{router} conditional edges \\
$\mathbf{X}_{k+1}=W_\sigma\mathbf{X}_k$      & \texttt{ring\_round} / \texttt{complete\_round} \\
$b_{k+1}=b_k-c_{\sigma_k}$    & \texttt{budget\_update} node \\
Thm.~\ref{thm:cert}         & \texttt{should\_continue} guard \\
$K^*$, $B^*$ (Thm.~\ref{thm:rac})& offline certificate before compile \\
\bottomrule
\end{tabular}
\end{table}

Through a text prompt, we assign the MAS a project where the agents must plan to launch an emergency medical tracking app. To increase initial disagreement and combat sycophancy, we exaggerate the urgency of the task, specifying a tiny infrastructure budget, top-secret data with maximum security, a two week launch period, and millions of requests. Then we assign roles to each of the agents through text prompts, giving each agent a specific priority (e.g. ``You are the Project Planner. You demand compromises in scalability to meet the deadline..." etc.). To preserve our assumption that each switching mode is consensus-contracting, we prompt the agents to adjust their preferences towards consensus at every round, without compromising their core values. 

At each communication round $k$, the agents communicate according to the topology activated under switching signal $\sigma_k.$ If agent $a_i$ communicates with agent $a_j$, we include the response which agent $a_i$ generated at round $k-1$ in the prompt of agent $a_j$ at round $k$. Each response is a text compilation of (i) the design-preference vector $x_i \in [0,1]^6 $ of agent $a_i$, (ii) the reasoning string which agent $a_i$ generated to justify its decision, and (iii) the influence weight which weight matrix $W_\sigma$ assigns to agent $a_i$. 

The influence weight of agent $a_i$ on agent $a_j$ is the entry $w_{ij}$ in the weight matrix $W_\sigma$. The weight matrices for Modes 1 and 2 are specified in Table \ref{tab:weights}. In a complete graph topology $K_5$, every agent communicates with every other agent. In a ring graph topology $C_5$, agent $a_i$ communicates with agent $a_{i-1}$ and agent $a_{i+1}.$ The relevant entries in each weight matrix, as listed in Table \ref{tab:weights}, were included in the prompt as a guideline for how much agent $a_j$ should incorporate $a_i$'s feedback. Agent $a_j$ updated its preference vector $x_j$ with no algebraic manipulation from the harness. Hence, the influence weights served as instructions for the agents rather than as means of direct modification.

We calculate disagreement $D(k)$ algebraically from the MAS aggregate belief matrix $\mathbf{X}_k$ using Definition \ref{def:disagreement}. To track token expenditure, we extract the number of input-tokens from the metadata of each agent's response. The ``input" of a response is the aggregate prompt, composed of the reasoning strings, design-preference vectors, and influence weights which agent $a_j$ receives from every agent $a_i$ that communicates with it. Hence, input-tokens accurately measure communication cost. We calculate $b_k$ at each round by subtracting the tokens expended at round $k$ from $b_{k-1}.$ 

We set the communication switching threshold $\eta = 0.2$ and the consensus tolerance $\varepsilon = 0.05.$ The initial token budget is $b_0 = $ 100,000. We track MAS disagreement and token expenditure across three strategies (i) adaptive, switching between Mode 1 and Mode 2 when $D(k) \leq 0.2$, (ii) complete-only (Mode 1), and (iii) ring-only (Mode 2). Except for differences in the parameters $\varepsilon, \eta$, these strategies are the exact same as the strategies from the simulations.

\begin{table*}[ht]
    \centering
     \caption{Execution summary across all models and strategies. We list the total number of communication rounds and total token expenditure from the initial disagreement $D(0)$ to the final disagreement $D_{final}.$}
    \label{tab:all_summary}
    \renewcommand{\arraystretch}{1.2} 
    \begin{tabular}{lccccccccc}
        \toprule
        & \multicolumn{3}{c}{\textbf{GPT-5.4-Mini}} & \multicolumn{3}{c}{\textbf{Gemini 2.5-Flash}} & \multicolumn{3}{c}{\textbf{Claude Opus-4.8}} \\ 
        \toprule
        & Adaptive & Complete & Ring & Adaptive & Complete & Ring & Adaptive & Complete & Ring \\ \bottomrule
        \textbf{Rounds} & 9 & 12 & 17 & 14 & 9 & 18 & 11 & 9 & 13 \\ 
        \textbf{Tokens} & 41,727 & 67,593 & 60,858 & 92,129 & 94,984 & 88,141 & 84,944 & 108,227 & 81,364 \\ 
        $\mathbf{D(0)}$ & 0.7778 & 0.8019 & 0.7570 & 0.7361 & 0.7158 & 0.7430 & 0.7240 & 0.6864 & 0.7193 \\ 
        $\mathbf{D_{final}}$ & 0.0409 & 0.0450 & 0.0285 & 0.0403 & 0.0430 & 0.0419 & 0.0456 & 0.0549 & 0.0401 \\ \bottomrule
    \end{tabular}
\end{table*}

\begin{figure*}[ht]
    \centering
    \begin{subfigure} [t] {0.49\textwidth}
        \caption{}
        \centering
        \includegraphics[width=1.0\linewidth, trim=0.8cm 0cm 1.8cm 1.7cm, clip]{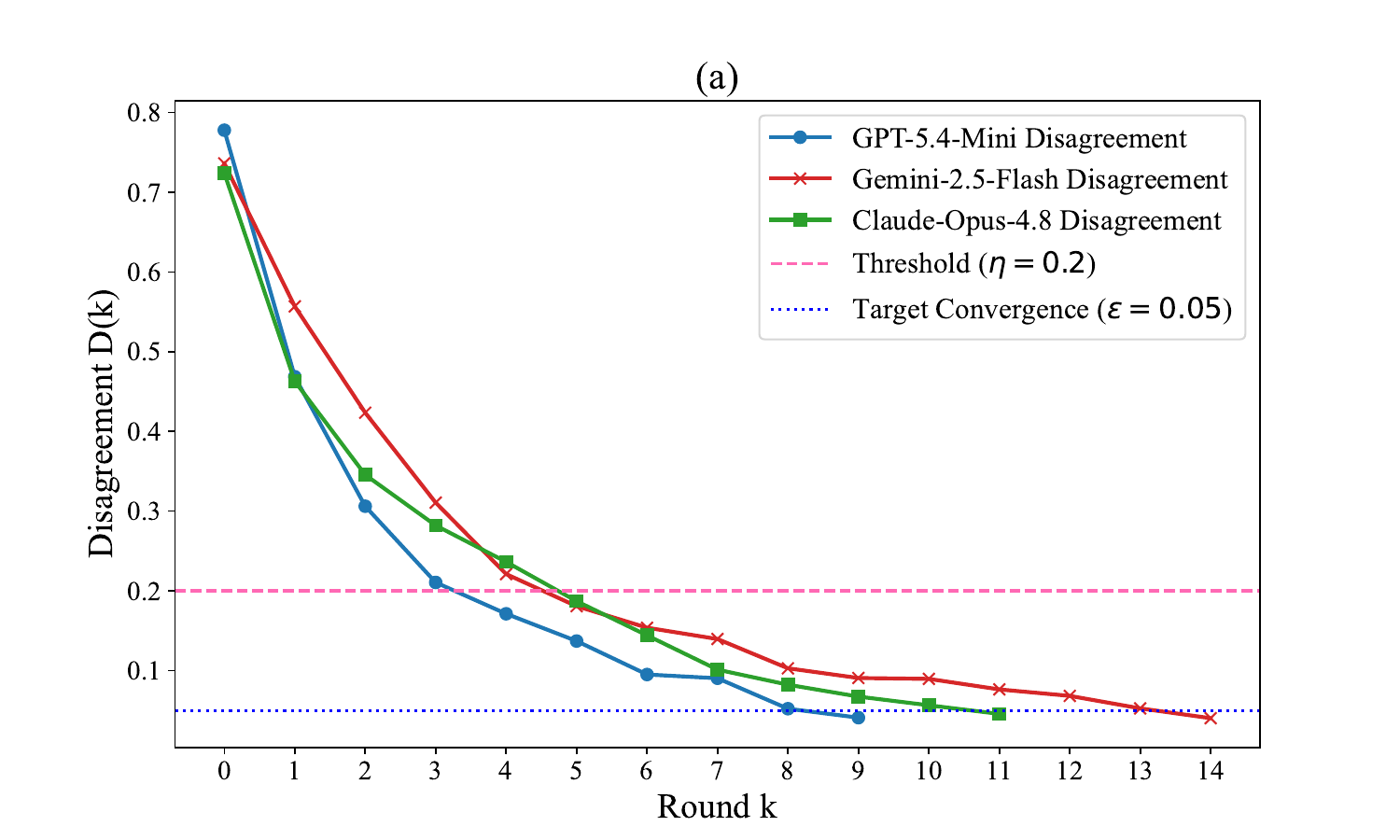}
        \label{fig:fig6}
    \end{subfigure}
        \hspace{0.1cm}
     \begin{subfigure} [t] {0.49\textwidth}
     \caption{}
        \centering
         \includegraphics[width=1.0\linewidth, trim=0.8cm 0cm 1.8cm 1.7cm, clip]{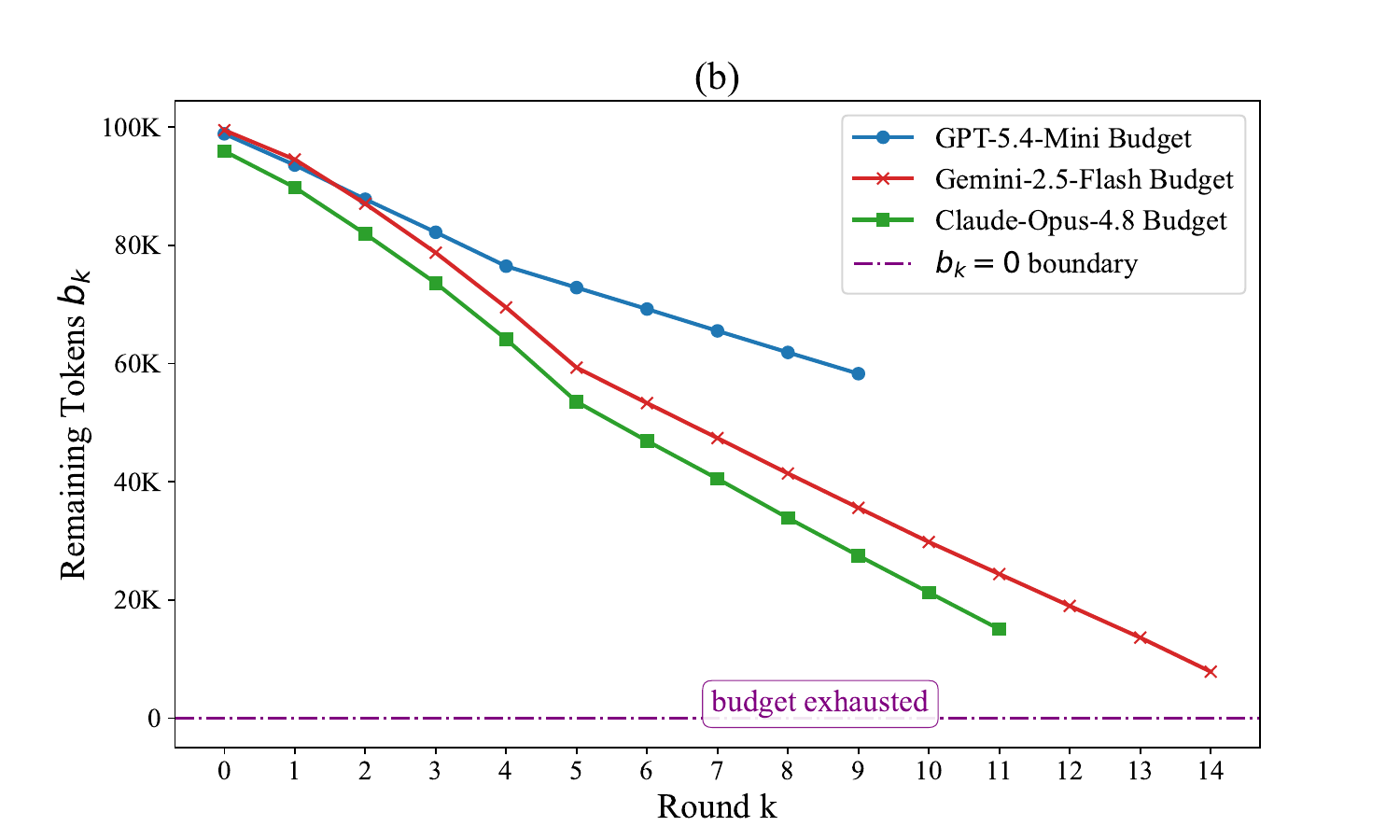}
       \label{fig:fig7}
    \end{subfigure}
     \begin{subfigure} [t] {0.49\textwidth}
     \caption{}
        \centering
        \includegraphics[width=1.0\linewidth, trim=0.8cm 0cm 1.8cm 1.7cm, clip]{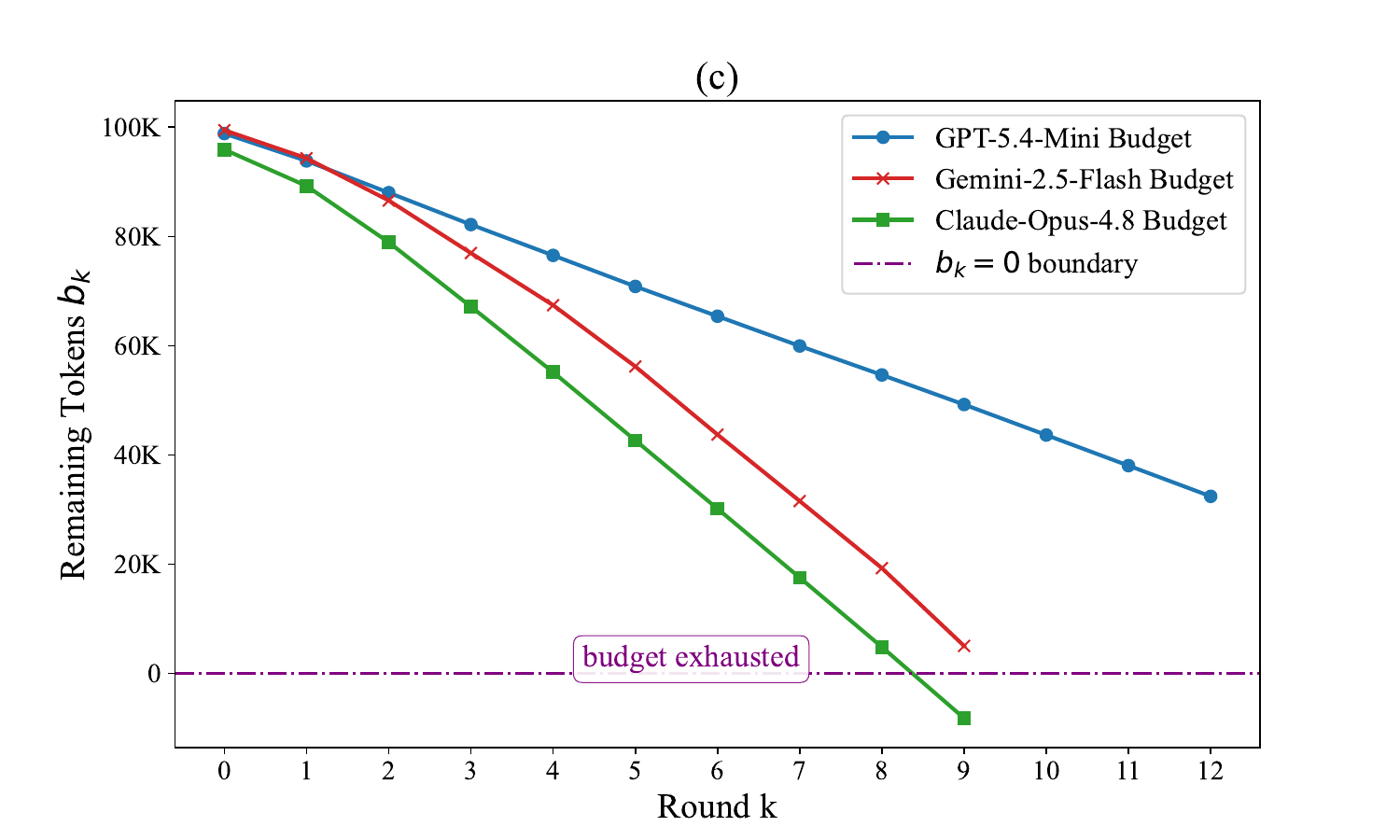}
        \label{fig:fig8}
    \end{subfigure}
        \hspace{0.1cm}
     \begin{subfigure} [t] {0.49\textwidth}
     \caption{}
        \centering
        \includegraphics[width=1.0\linewidth, trim=0.8cm 0cm 1.8cm 1.7cm, clip]{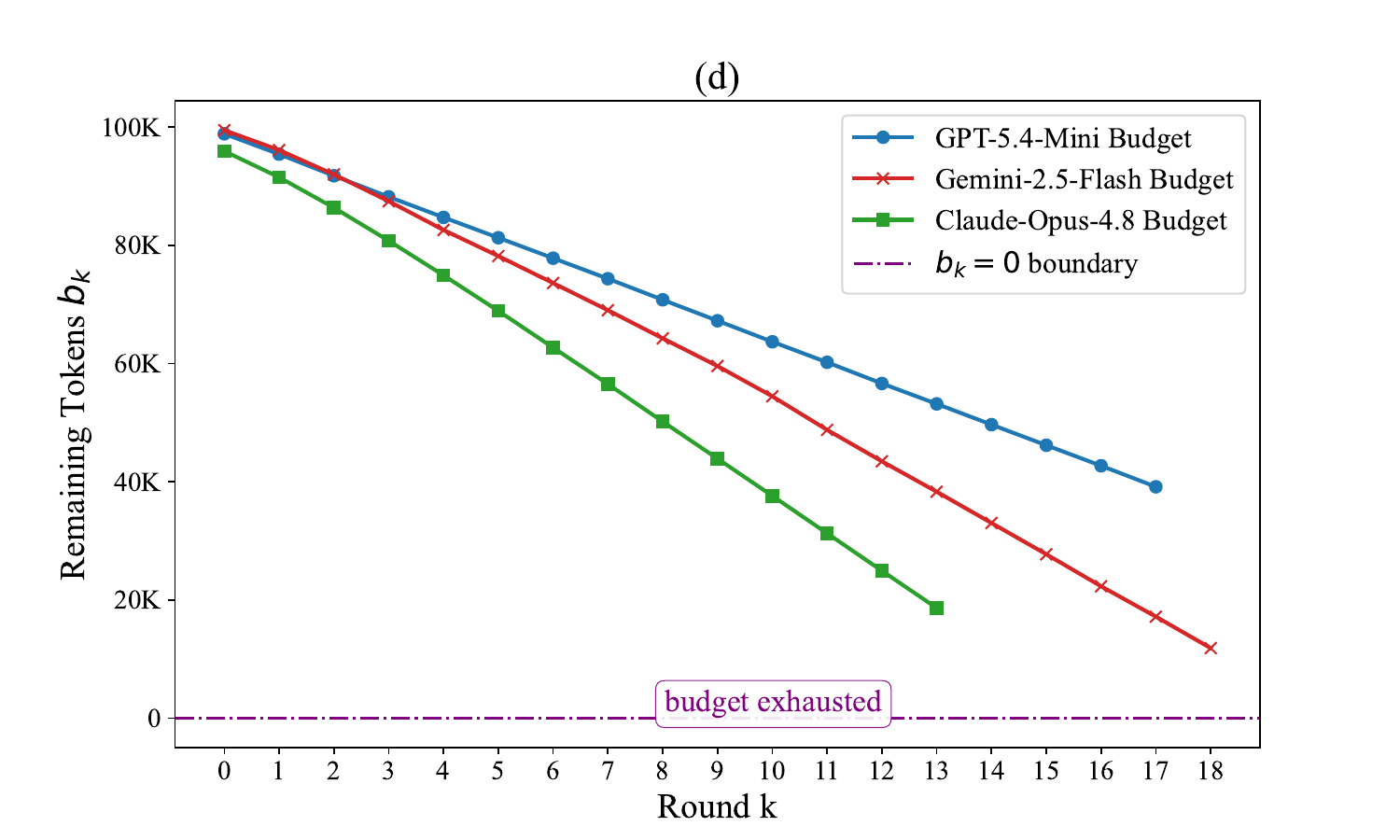}
            \label{fig:fig9}
    \end{subfigure}
    \caption{Analysis of budget and disagreement dynamics across different architectural strategies. Panels (a) and (b) examine the adaptive strategy, while (c) and (d) examine token costs for the complete-only and ring-only strategies.}
    \label{fig:main_four_plots}
\end{figure*}

 We summarize the executions across strategies and choice of LLM in Table \ref{tab:all_summary}. In the adaptive and ring strategies, the MAS debate sequences reflect our assumption that each communication mode is consensus-contracting, as $D(k)$ decreases with every round. It is worthwhile to mention that, during tests where we did not explicitly instruct the agents to compromise on each round, they would often increase $D(k)$, leading to significantly more rounds, or even budget failures. 

We noticed that the faster GPT-5.4-Mini model used less tokens than the Gemini 2.5-Flash and the Claude Opus-4.8. For the Gemini deployments, we scaled all token expenditures by a constant 0.75 for both the complete and ring modes, so that the Gemini deployment could fit within the same 100,000 token budget as the GPT deployment. The method for counting tokens is unique to the LLM, and we are not interested in comparing expenses across models. Instead, we are identifying behavioral traits of MAS that extend across models, and so a constant scaling merely makes these traits easier to see. 

Anthropic's Claude Opus-4.8 consumed drastically more tokens than Gemini and GPT, so we scaled all token costs by 0.45. This scaling allowed us to plot the token expenditures from all three deployments on the same graph. Note that we applied the same constant scalings across deployments under the adaptive, complete-only, and ring-only strategies. 

Fig. \ref{fig:fig6} plots the disagreement level for all three models under the adaptive strategy. Disagreement levels decline in two distinct stages under the adaptive strategy. When Mode 1 is activated ($D(k)> \eta$), global communication brings the MAS toward consensus rapidly, and the slope of the $D(k)$ curve is steep. In Mode 2, limited communication under the ring topology causes slower progress towards consensus, with a correspondingly gentle slope. All three LLMs exhibit this two-stage behavior before eventually crossing the dotted $\varepsilon$ threshold and reaching consensus.
 The Gemini MAS takes 14 rounds to reach consensus, while the GPT agents agree much more quickly after 9 rounds. The Claude agents reach consensus after 11 rounds. Under the adaptive strategy, none of the  models exit $\Omega_\varepsilon$ by triggering \textsc{BudgetFail}. 

In Fig. \ref{fig:fig7}, we plot budget dynamics under the adaptive strategy. For all three MAS, switching communication modes precipitates a decisive change in token costs. Because Mode 1 is significantly more expensive, the slopes of the budget plots are very steep for the early rounds. After the system passes the switching threshold (round 4 for GPT, round 5 for Claude and Gemini), the cost per round decreases under Mode 2.

Table \ref{tab:all_summary} summarizes the executions for the complete-only strategy. We show the budget trajectories in Fig. \ref{fig:fig8}. Token expenditures increased significantly for all three models, and the Claude MAS triggered \textsc{BudgetFail} on the final round and exited the safe set $\Omega_\varepsilon$. While the Gemini MAS reached consensus faster (9 rounds under complete-only vs. 15 under the adaptive strategy), GPT took longer to reach consensus (12 rounds under complete-only vs. 9 under adaptive). Differences in how the LLMs processed feedback from fellow agents could account for this difference. Gemini agents found the extra communication afforded by the complete topology helpful for aligning their positions to the rest of the group, and so they reached consensus faster. On the other hand, feedback from the entire group only confused and hampered the GPT agents once they had reached the point where slight adjustments in their belief vectors were needed. The logs from this run show that the GPT MAS disagreement level rose after rounds 7 and 9 (rounds which would have used the ring topology (Mode 2) under the adaptive strategy). Even with our instruction to compromise at every round, the consensus contraction rates for these rounds were $> 1.$ While we derived boundary expressions for the simple case where consensus rates are strictly contracting, the dynamics of real MAS are often more complicated. 

In Fig. \ref{fig:fig9}, we show the budget trajectories under the ring-only strategy; we list the total execution summary for all three models in Table \ref{tab:all_summary}. Due to the cheaper cost of Mode 2, the MAS reach consensus within the budget for all three LLMs. The number of rounds required for this strategy, however, is significantly higher than for the previous two strategies: 17 for the GPT MAS, 18 for the Gemini MAS, and 13 for Claude the MAS. These result aligns with out expectations from the simulations: Mode 2 uses less tokens than Mode 1, but takes longer to reach consensus. Due to the increased number of rounds, the token savings from the ring-only strategy were minimal for the GPT and Gemini MAS.

These experiments validate the intuition given by Theorems~\ref{thm:rac}--\ref{thm:cert} and the earlier simulations. Unsafe behavior occurs when $(\mathbf{X}_0, b_0 )\notin \Omega_{cert}$, as with the Claude MAS under the complete-only strategy. For that deployment, the consensus-contraction rate for every round $\lambda_{\sigma_k}$ remained roughly constant, with an average value of $\lambda_1 \approx 0.7574$ across all 9 rounds. Since this experiment involved Mode 1 only, $\eta = \varepsilon = 0.05 \implies K^* = K_1 \approx 10$. Token costs per round also remained roughly constant, and we can average them to estimate $c_1 \approx $ 12,025. Then $B^*(D(0)) \approx 10 \cdot 12025 = 120250 > 100000  = b_0$. Thus, $(\mathbf{X}_0, b_0) \notin \Omega_{cert}$. As Theorem \ref{thm:cert} predicts, $(\mathbf{X}_9, b_9) \notin \Omega_{\varepsilon}$, and so the MAS triggered \textsc{BudgetFail}.

Debate among real LLMs introduces dynamics not captured by the algebraic surrogate step \eqref{eq:consensus}. The fact that the GPT MAS achieved consensus with the least time and expense under the adaptive strategy suggests that the level of disagreement changes the consensus-contraction rate of a communication topology. When the system is close to consensus, only slight adjustments to each agent's preference vector $x_i$ are needed. At this point, too much information sharing among the agents can be counter-productive because it leads to unnecessarily drastic modifications. The agents are trying to accommodate the entire team, when they only need to adjust based on their nearest neighbors. 

The Claude MAS \textsc{BudgetFail} in the complete-only deployment demonstrates how a competent switching policy is necessary to steer the MAS trajectory into the safe set $\Omega_\varepsilon.$ Under the ring-only strategy, Gemini MAS needed 18 rounds to achieve consensus, with only slight token savings. The adaptive switching policy, however, reaped performance benefits in both speed and cost.


\section{Conclusions}\label{sec:conclusion}

We presented a switched system framework for LLM multi-agent
consensus with norm-threshold switching, a forward-invariance theorem, budget safety certificate, and a Resource-Aware Consensus
Theorem (\ref{thm:cert}) linking spectral contraction rates, consensus time, and minimum
token budget.
Corollary~\ref{cor:tradeoff} certifies when the adaptive policy
Pareto-dominates single-topology strategies.
Live deployments of LLMs in a realistic software-architecture planning workflow
($N=5$ agents, $d=6$ design dimensions) illustrates all theoretical
predictions and demonstrates the efficiency gain of adaptive
communication over both ring-only and complete-only alternatives.

Future directions include dynamic consensus-contraction rates and communication costs. The preceding analysis treats the communication costs
$c_1$ and $c_2$ as fixed constants. In practical multi-agent AI systems, however, both communication cost and collective disagreement depend on the number of agents participating in the deliberation process. As the number of agents $N$ increase, communication complexity and costs also increase. Examining the competing mechanisms which arise in larger teams would be an important step in generalizing the results we present here. 

The consensus-contraction rate $\lambda_\sigma$ of a given communication mode $\sigma$ depends on $N$. A larger team can contribute additional perspectives, specialized knowledge, and alternative reasoning paths. However, these benefits are accompanied by greater heterogeneity in the initial belief state, and increased difficulty in reaching consensus. 

As we saw in the live LLM deployments, the disagreement level $D(k)$ also effects $\lambda_{\sigma_k}.$ Developing an optimal switching policy that balances the complex relationships between team size $N$, disagreement level $D(k)$, consensus-contraction rate $\lambda_{\sigma_k}$, and communication cost $c_{\sigma_k}$ would extend the concept of ``safe" MAS beyond simplified scenarios. 

Developing a method for estimating $\lambda_\sigma$ and $c_\sigma$ for a given MAS configuration would be an essential step for applying our formulations of safe regions in real time. In this work, we estimated parameters by averaging their values for a given deployment, and then reverse-engineered the boundary conditions for $\Omega_{cert}$. In actual practice, however, one would need to know these parameters before the deployment to predict whether the MAS would remain within $\Omega_{\varepsilon}$ given an initial disagreement $D(0)$ and budget $b_0.$ We expect that team size $N$, choice of LLM, and the pre-set temperature of the LLM would heavily effect $\lambda_\sigma$ and $c_\sigma.$

\end{document}